\documentclass[12pt,oneside,reqno]{amsart}

\usepackage[a4paper,left=24mm,top=29mm,right=24mm,bottom=29mm]{geometry}

\usepackage{amssymb,amsfonts}
\usepackage[all,arc]{xy}
\usepackage{enumerate}
\usepackage{siunitx}
\usepackage{enumitem}
\usepackage{mathtools}
\usepackage{mathrsfs}
\usepackage{dsfont}
\usepackage{amsfonts}
\usepackage{amssymb}
\usepackage{graphicx}
\usepackage{amsthm}
\usepackage{amsmath}
\usepackage[mathscr]{euscript}
 \let\mathscr\relax
\usepackage[scr]{rsfso}
\usepackage{graphicx}
\usepackage{color}
\usepackage{float}
\usepackage{caption}
\usepackage[pdftex,bookmarksnumbered,bookmarksopen]{hyperref}
\hypersetup{hidelinks = true}

\DeclareMathOperator{\Tr}{Tr}
\DeclareMathOperator{\vol}{Vol}
\DeclareMathOperator{\Var}{Var}
\DeclareMathOperator{\Cov}{Cov}
\DeclareMathOperator{\re}{Re}
\DeclareMathOperator{\im}{Im}
\DeclareMathOperator{\Corr}{Corr}
\DeclareMathOperator{\Arg}{Arg}

\newtheorem{thm}{Theorem}[section]

\newtheorem{prop}[thm]{Proposition}
\newtheorem{lem}[thm]{Lemma}

\newtheorem{claim}[thm]{Claim}

\theoremstyle{definition}

\theoremstyle{remark}
\newtheorem{rem}[thm]{Remark}

\numberwithin{equation}{section}

\begin{document}
		\title[Spectral quasi correlations and Arithmetic Random Waves]{ Spectral quasi correlations and phase-transitions for the nodal length of Arithmetic Random Waves }
	\begin{abstract}
		We study the nodal length of Arithmetic Random Waves at small scales: we show that there exists a phase-transition for the distribution of the nodal length  at
		a logarithmic power above Planck-scale. Furthermore, we give strong evidence for the existence of an intermediate phase between Arithmetic and  Berry's random waves. These results are based on the study of small sums of lattice points lying on the same circle, called spectral quasi correlations.  We show that, for generic integers representable as the sum of two squares, there are \textit{no} spectral quasi correlations. 
	\end{abstract}
	\author{Andrea Sartori}
\address{ Departement of Mathematics, King's College London, Strand, London WC2R 2LS, England, Uk }
\email{andrea.sartori.16@ucl.ac.uk}
\maketitle	
\section{Introduction}
\subsection{The Random Wave Model and the nodal length of Laplace eigenfunctions}
\label{phase-transition}
 Given a compact Riemannian surface $(M,g)$ without boundary, let $\Delta_g$ be the Laplace-Beltrami operator on $M$. There exists an orthonormal basis for $L^2(M,d\vol)$ consisting of eigenfunctions $\{f_{\lambda_i}\}$	
\begin{align}
\Delta_g f_{\lambda_i}+ \lambda_i f_{\lambda_i}=0  \nonumber
\end{align}
with $0=\lambda_1<\lambda_2\leq...$ listed with multiplicity, and $\lambda_i\rightarrow \infty$. One of the main  characteristics of an eigenfunction $f_{\lambda}$ is its \textit{nodal set} 
\begin{align}
\mathcal{Z}(f_{\lambda})= \{ x \in M: f_{\lambda}(x)=0\}. \nonumber
\end{align}
 It is known that $\mathcal{Z}(f_{\lambda})$ is the union of smooth curves outside a finite set of points \cite{C} and  Yau conjectured that its volume, the \textit{nodal length},  satisfies 
\begin{align}
c\sqrt{\lambda}\leq\mathcal{L}(f):=\mathcal{H}(\mathcal{Z}(f_{\lambda}))\leq C\sqrt{\lambda} \label{Yau}
\end{align}
for some constants $c,C>0$ which depend on $M$ only, where $\mathcal{H}(\cdot)$ is the Hausdorff measure. Donnelly and Fefferman \cite{DF} showed that Yau's conjecture holds for any real-analytic manifold (of any dimension), and recently, Logunov and Malinnikova \cite{LM,L1,L2} proved the lower-bound in the smooth case and gave a polynomial upper-bound.

Berry \cite{B1,B2} conjectured  that \textquotedblleft generic\textquotedblright \hspace{1mm} Laplace eigenfunctions $f_{\lambda}$ can be modelled,  in balls of radius slightly larger than  $O(\lambda^{-1/2})$, the \textit{Planck-scale}, by monochromatic plane waves, an isotropic  Gaussian field with the spectral measure the Lebesgue measure on the unit circle. This Gaussian field is also known as Berry's Random Waves (BRW). In particular, Berry's model suggests that \textquotedblleft generic\textquotedblright \hspace{1mm} Laplace eigenfunctions change their behaviour when restricted to sufficiently small balls, we are interested in exploring how this affects their nodal length.

\subsection{Phase-transitions for the nodal length of Arithmetic Random Waves}
\label{phase-transitions ARW}
We study \textit{random} Laplace eigenfunctions on the flat two dimensional torus $\mathbb{T}^2=\mathbb{R}^2/\mathbb{Z}^2$, also known as Arithmetic Random Waves (ARW). These are Gaussian random fields satisfying
$$\Delta f_n +4\pi^2 n f_n=0$$
where $\Delta$ is the flat Laplacian and the eigenvalue $n\in S:=\{ n=a^2+b^2: \hspace{2mm} a,b \in \mathbb{Z}\}$, the set of integers representable as the sum of two squares. Explicitly, $f_n$ can be defined as
\begin{align}
f_n(x)= \frac{1}{\sqrt{N}}\sum_{\substack{\xi \in \mathbb{Z}^2 \\|\xi|^2=n}}a_{\xi}e(\langle \xi, x \rangle) \label{function}
\end{align}
where $e(\cdot)= e^{2\pi i \cdot}$,  $a_{\xi}$ are i.i.d. standard complex Gaussian random variables save for $\overline{a_{\xi}}= a_{-\xi}$ so that $f_n$ is real valued, and the normalisation  constant $N=r_2(n)$, the number of lattice points on the circle of radius $\sqrt{n}$, in \eqref{function} implies that $\mathbb{E}[|f_n|^2]=1$. Up to rescaling $\mathbb{T}^2$, ARW can equivalently be defined, via Kolmogorov's Theorem, as the centred, stationary, Gaussian random field with spectral measure 
\begin{align}
\mu_n= \frac{1}{N}\sum_{|\xi|^2=n}\delta_{\xi/\sqrt{n}} \label{measure}
\end{align}
where $\delta_{\xi/\sqrt{n}}$ is the Dirac distribution at $\xi/\sqrt{n}$. 

The study of  the nodal length of the ARW was initiated by Oravecz, Rudnick and Wigman \cite{ORW}; Rudnick and Wigman \cite{RW} found the expectation of $\mathcal{L}(f_n)$ and gave an upper bound for the variance. Subsequently, 	Krishnapur,  Kurlberg and Wigman \cite{KKW} proved that
\begin{align}
\Var [\mathcal{L}(f_n)]= \frac{1+ \hat{\mu}_n(4)}{512}\frac{n}{N^2}\left(1+o_{N\rightarrow \infty}(1)\right) \label{variance f_n}
\end{align}
where $\hat{\mu}_n(4)$ is the fourth Fourier coefficient of the measure $\mu_n$. Notably, the accumulation points of the sequence $\hat{\mu}_n(4)$ contain the interval $[0,1]$, \cite{CI,KW,SA2}.  Finally, a non-universal, non-central limit law for $\mathcal{L}(f_n)$ was found by Marinucci, Peccati, Rossi and Wigman \cite{MPRW}.

 Berry \cite{B2} showed that $$\Var[\mathcal{L}(f_{\mu})]= \frac{1}{256\pi}\log n(1+o_{n\rightarrow \infty}(1))$$
 where $f_{\mu}$ are BRW (in a square of side $1$), $\mu$ is the Lebesgue measure on the unit circle and, for the sake of consistency, $n$ represent the eigenvalue. Therefore, the asymptotic expansion \eqref{variance f_n} shows that the \textit{total} nodal length of the ARW behaves differently than the nodal length of the BRW. However, since the spectral measure $\mu_n$ converges to $\mu$ for almost all $n\in S$ \cite{EH,KK}, the field $f_n$ generically behaves like BRW in balls of radius $O(n^{-1/2})$. This suggests the existence of a phase-transition in the behaviour of the nodal length of \textquotedblleft generic\textquotedblright  ARW and, in analogy with the study of the ARW at small scales in \cite{GW}, we expect the said phase-transition to happen at some logarithmic power above Planck-scale.

Investigating the nodal length of ARW at small scales, Benatar, Marinucci and Wigman \cite{BMW}, suggested that there actually exits a intermediate phase between ARW and BRW: letting
\begin{align}
\mathcal{L}(f_n,s)= \vol \{ x\in B(s): f_n(x)=0\} \nonumber
\end{align} 
where $B(s)$ is the ball of radius $s$ centred at the origin, it is expected that there exists some exponent $A_0>0$ such that the law of $\mathcal{L}(f_n,s)$ agrees with the law of $\mathcal{L}(f_n)$ for $s>(\log n)^{A_0}/n^{1/2}$ and behaves differently for $s< (\log n)^{A_0}/n^{1/2}$ .  In this direction, Benatar, Marinucci and Wigman \cite{BMW} found that, for a density one subsequence of $n\in S$, the asymptotic expansion of the variance of $\mathcal{L}(f_n,s)$ agrees, once appropriately rescaled, with \eqref{variance f_n} and they deduced that $\mathcal{L}(f_n,s)$ fully correlates with $\mathcal{L}(f_n)$, provided that $s>n^{-1/2+\epsilon}$. Moreover, they related  $\mathcal{L}(f_n,s)$ to the notion of \textit{spectral quasi-correlations}, while $\mathcal{L}(f_n)$ is related to the notion of  \textit{spectral correlations} \cite{BB, KKW}. We are now going to describe spectral correlations and quasi-correlations and we will give the details of such relations in section \ref{ARW and correlations} below.

\subsection{Statements of main results}
 The set of \textit{spectral correlations} is
\begin{align}
	\mathcal{S}(l,n):=\{(\xi_1,...,\xi_l): \xi_1+...+\xi_l=0 \hspace{3mm} |\xi_i|^2=n\} \label{1}
\end{align}
 where $l$ is a positive integer, the \textit{length} of the correlations, $n\in S$ and $\xi_i$ are the lattice points on the circle of radius $\sqrt{n}$.  While the set of \textit{spectral quasi-correlations} is
\begin{align}
\mathcal{Q}(l,n,K):=\{(\xi_1,...,\xi_l): 0<| \xi_1+ \xi_2 + \xi_3+...+ \xi_l|\leq K \hspace{3mm} |\xi_i|^2=n\} \label{4}
\end{align}
where $K>0$ is some parameter. Importantly, $\mathcal{Q}(2l,n,K)$ excludes the set of \textquotedblleft diagonal\textquotedblright solutions $\xi_1=-\xi_2,...,\xi_{2l-1}=-\xi_{l}$ which is contained in $\mathcal{S}(2l,n)$. To study ARW, we are interested in the largest $K$ such that $\mathcal{Q}(l,n,K)= \emptyset$ as $n\rightarrow \infty$. 
   
 Harman and Lewis \cite{HL} showed that there are infinitely many primes of the form $p=a^2+b^2$ with $|b|\leq p^c$ for some small constant $0<c<0.119$. For the such primes, there are two lattice points, $\xi_{\pm}=(\pm a,b)$, with $|\xi_{+}+ \xi_{-}|=b$, thus $\mathcal{Q}(2,p,p^{c})\neq \emptyset$. Moreover, if there exist infinitely many primes of the form $p=m^2+1$, as it was conjectured by Landau, then  $\mathcal{Q}(2,p, O(1))\neq \emptyset$. However, if we consider a generic integer $n\in S$, we can prove the following: 
\begin{thm}
	\label{theorem 1}
	Let $l\geq 2$ be an integer, $\epsilon>0$ and define $c(l)=c(l,\epsilon)$ recursively as follows: $c(2)= \log 2+\epsilon$, $c(3)= 3\log 2/2+\epsilon$ and $c(l)= l\log 2/2 +c(\lfloor l/2 \rfloor) +\epsilon$, where  $\lfloor l/2 \rfloor $ represent the largest integer smaller than $l/2$. Then  for almost all $n\in S$, we have 
$$
	\mathcal{Q}(l,n, n^{1/2} / (\log n)^{c(l)})= \emptyset. 
	$$
\end{thm}

It is also relevant to us to investigate when $\mathcal{Q}(l,n,K)\neq \emptyset$. Erd\"{o}s-Hall \cite[Theorem 3]{EH} showed that, for almost all $n\in S$,  $\mathcal{Q}(2,n,  n^{1/2}\log n ^{\log3/2 + \epsilon})\neq \emptyset$. Therefore, by the triangle inequality, we have
\begin{align}
\mathcal{Q}(2l,n,l \cdot n^{1/2}/ (\log n)^{\log3/2+\epsilon})\neq \emptyset \nonumber
\end{align}
for almost all $n\in S$ and, in section \ref{small spectral} below, we will show that also odd length quasi-correlations can be small, in the appropriate sense. 

In order to understand the size of $\mathcal{Q}(l,n,K)$, we study a random model for a generic integer $n\in S$, see also \cite[Remark 3.3]{GW}. In this model, the angles of the Gaussian primes dividing $n$ are represented by i.i.d. uniform random variables on $[0,2\pi)$ so that the lattice points $\xi_i$ are random variables taking values on the circle of radius $\sqrt{n}$,  more details will be given in sections \ref{NT background} and \ref{Random model} below. We then define the random sums $X_{\underline{i}}:=(\xi_{i_1}+...+\xi_{i_l})/n^{1/2}$ where $\underline{i}=(i_1,...,i_l)$ for $1\leq i_j\leq N$ and prove the following:
\begin{thm}
	\label{thm random}
	Let $l\geq 2$, $n\in S$, $0<\alpha<1$ be some parameter which may depend on $n$ and $X_{\underline{i}}$ for  $\underline{i}=(i_1,...,i_l)$ be as above. Then we have
	\begin{align}
	\mathbb{E}[\#\{X_{\underline{i}}: |X_{\underline{i}}|\leq \alpha\}]\asymp_l (\alpha + O_l\left(\alpha^2\right)) N^{l}(1+o_{N\rightarrow \infty}(1)) \nonumber
	\end{align} 
	where $A\asymp B$ if there exist two constants $c,C>0$ such that $c A\leq B\leq C A$ and  the constants implied in the notation depend on $l$ only.
\end{thm}
We observe that, taking $\alpha=
N^{-l}$ in Theorem \ref{thm random}, we expect $Q(l,n, O( n^{1/2}N^{-l}))\neq \emptyset$. Since, for almost all $n\in S$,  $N\asymp (\log n)^{\log 2/2 \pm\epsilon}$, see Lemma \ref{ERKA} below, and $c(l)\leq 2l$, Theorem \ref{thm random} suggests that Theorem \ref{theorem 1} gives the right order of growth, in $l$, for the constant $c(l)$.

Thanks to Theorem \ref{theorem 1}, following similar techniques to \cite{BMW}, we are able to prove the following upper bound for the phase-transitions: 
\begin{thm}
	\label{thm 2}
	Let $A= \frac{29}{6} \log 2=3.3512...$ and $\epsilon>0$. There exists a density one subsequence of $n\in S$ such that the following holds: 	
	\begin{enumerate}
		\item $N(n)\rightarrow \infty$ and the set of accumulation points of $\{ \hat{\mu}_n(4)\}$ contains $[0,1]$. 
		\item Uniformly for $s>(\log n)^{A+\epsilon} \cdot n^{-1/2}$,  we have 
		\begin{align}
		\Var( \mathcal{L}(f_n,s))= \frac{1+ \hat{\mu}_n(4)}{512} (\pi s^2)^2 \frac{n}{N^2}\left(1 + o_{N\rightarrow \infty}(1)\right) .\nonumber
		\end{align}
		\item We have
		\begin{align}
		\sup_{s>(\log n)^{A+\epsilon}/n^{1/2} } |\Corr( \mathcal{L} (f_n,s) \nonumber
		, \mathcal{L}(f_n))- 1| \rightarrow 0
		\end{align}
		where, $\Corr(X,Y)= \Cov(X,Y)/ (\Var(X))^{1/2} (\Var(Y))^{1/2}$. 
	\end{enumerate}
\end{thm} 

Given a sequence of $n\in S$ that satisfies the conclusion of Theorem \ref{theorem 1} and \eqref{3}, part $(2)$ and part $(3)$ of Theorem \ref{thm 2} follow directly using the techniques in \cite{BMW}. However, \textit{a priori}, for any said sequence, $\hat{\mu_n}(4)$ might have only one accumulation point. To rule this out, we explicitly construct sequences of $n\in S$, satisfying \eqref{3} and the conclusion of Theorem \ref{theorem 1}, for which we can control the distribution of lattice points on  $\sqrt{n}S^1$. Benatar, Marinucci and Wigman's argument relies on density estimates, thus our result seems to be the first to give explicit examples of such sequences in the literature.
 
	Finally, we show that there exists some $B>0$ such that $\mathcal{L}(f_n,s)$ behaves like the nodal length of BRW for $s<(\log n)^{B} \cdot n^{-1/2}$. This provides a lower bound for the phase-transitions and shows that the behaviour of the nodal length changes at some logarithmic power above the Planck scale. 
	\begin{thm}
		\label{thm 3}Let $B=\frac{1}{84}\log \frac{\pi}{2}=0.0053...$, $\epsilon>0$ and $R>1$. Moreover, let $F_n(\cdot)=f_{n}(R\cdot /\sqrt{n})$ and $F_{\mu}(\cdot)=f_{\mu}(R\cdot)$, where $f_{\mu}$ is the BRW. There exists a density one subsequence of $n\in S$ such that for all $R\leq \log n^{B -\epsilon}$ and all fixed $t\in(-\infty, \infty)$, we have 
		\begin{align}
		&\left| \mathbb{E}[\exp\left( it\mathcal{L}(F_n)\right)]- \mathbb{E}[\exp\left(it \mathcal{L}(F_{\mu})\right)] \right|\longrightarrow 0 & n\rightarrow \infty. \nonumber
		\end{align} 
	\end{thm}
 As mentioned in section \ref{phase-transitions ARW}, $\Var[\mathcal{L}(f_{\mu})]$ was computed by Berry \cite{B2} and Wigman \cite{W} found the variance for the nodal length of random spherical harmonics on the two dimensional sphere. The law, for random spherical harmonics, was discovered by Marinucci, Rossi and Wigman \cite{MPW}. Subsequently, Nourdin, Peccati and Rossi \cite{NPR} found the law of $\mathcal{L}(f_{\mu})$ to be Normal, once appropriately normalised. 
\subsection{Intermediate phase}  
\label{ARW and correlations}
Let $f_n$ be as in \eqref{function} and $s>0$ be some parameter. Thanks to the Kac-Rice formula, moments of $\mathcal{L}(f_n,s)$ can be expressed in terms of the \textit{restricted}  moments of covariance function
\begin{align}
 r_n(x,y)= \int_{S^1} e(\langle x-y, \lambda\rangle)d\mu_n(\lambda) \label{covariance}
 \end{align}
 where $S^1\subset \mathbb{R}^2$ is the unit circle. That is, for $l\geq 2$ we are interested in asymptotically evaluating 
\begin{align}
\int_{B(s)} r(x)^ldx= \frac{1}{N^l}\sum_{\xi_1,...,\xi_l} \int_{B(s) }e(\langle \xi_1+...+\xi_l,x \rangle)dx \nonumber
\end{align}
where $B(s)$ is the ball centred at $0$ of radius $s$. Separating the terms with $\xi_1+...+\xi_l=0$, we obtain
\begin{align}
\int_{B(s)}r(x)^l dx= \pi s^2\frac{\#\mathcal{S}(l,n)}{N^l} + \frac{2\pi s^2}{N^l}\sum_{|\xi_1+...+\xi_l|>0}\frac{J_1(s|\xi_1+...+\xi_l|)}{s|\xi_1+...+\xi_l|} \label{6}
\end{align}
where $\mathcal{S}(l,n)$ is as in  \eqref{1} and $J_1(\cdot)$ is the Bessel function of the first kind.

Spectral correlations have been studied by Bombieri and Bourgain \cite{BB} who showed that 
\begin{align}
	\#\mathcal{S}(2l,n)= \frac{(2l)!}{l! \cdot 2^l}N^l(1+o(1))  \label{3}
\end{align}
for almost all $n\in S$, while $\mathcal{S}(l,n)= \emptyset$ for $l$ odd by congruence obstruction modulo $2$,  see also \cite{KKW} and section \ref{NT background} below for a more detailed discussion. Since $J_1(T)\ll T^{1/2}$ for $T$ large enough, the second term in \eqref{6} would asymptotically vanish if $\mathcal{Q}(l,n,s^{-1})=\emptyset$, which is in particular the case if $s=O(1)$. Hence, the phase-transition for $\mathcal{L}(f_n,s)$ can be compared to the change in the asymptotic law of the second term in \eqref{6}.

In light of Theorem \ref{thm random}, we expect that $A_0=2\log 2=1.3862...$. Indeed, thanks to the calculations in \cite{BMW}, we need to control the asymptotic in \eqref{6} only for $l=2,4,6$. Thus, given some $l\geq 2$, Theorem \ref{thm random}, with $\alpha=N^{-l/2}$, suggests that there are at most $o(N^{l/2})$ tuples $(\xi_1,...,\xi_l)$ such that $s|\xi_1+...\xi_l|< N^{l/6}$, where $s=N^{2l/3}/n^{1/2}$. For the remaining $l$-tuples, we have that $(J_1(s|\xi_1+...+\xi_l|)/s|\xi_1+...+\xi_l|)^2=o(N^{l/2})$. Thus, using the bound $J_1(T)/T=O(1)$ for the former set of tuples and the bound $J_1(T)/T\ll T^{-3/2}$ for the latter,  the  second term in \eqref{6} is negligible compared to the first one. Taking $l=6$ and bearing in mind that $N\asymp \log(n)^{\log 2/2\pm\epsilon}$, we obtain $s= (\log n)^{2\log 2 +\epsilon}/n^{1/2}$. 

\subsection{Related work}
Bourgain and Rudnick \cite{BR} first studied length $2$ quasi-correlations and showed that  $\mathcal{Q}(2,n, n^{1/2-\epsilon})=\emptyset$ for a density one subset of $n\in S$. Subsequently, Granville and Wigman \cite{GW} showed that
\begin{align}
\#\{n\in S: n\leq X \hspace{2mm} \text{and} \hspace{2mm} \mathcal{Q}(2,n,K) \neq \emptyset \}= C\sqrt{X}K\left( (2\log K)^{1/2}+O(1)\right), \nonumber
\end{align}
for some explicit $C>0$, this implies that  
\begin{align}\label{1.1}
\mathcal{Q}(2,n, n^{1/2}/\Psi(n)\log n)=\emptyset
\end{align}
for almost all $n\in S$ and any function $\Psi(n)\rightarrow \infty$ as $n\rightarrow \infty$. Theorem \ref{theorem 1} refines \eqref{1.1} to $\mathcal{Q}(2,n, n^{1/2}/\log n^{\log 2+\epsilon})=\emptyset$. This is also directly related to the question of estimating the number of lattice points $\xi_i$ on an arc of length $n^{\delta}$ for $\delta>0$. Cilleruelo and Cordoba \cite{CC} showed that there are at most $O_\delta(1)$ such lattice points if $\delta<1/4$ and Cilleruelo and Granville \cite{CG} conjectured that this remains true for every $\delta<1/2$.  Theorem \ref{theorem 1} implies that there are at most $2$ lattice points on any arc of length at most $n^{1/2}/\log n^{\log 2+\epsilon}$ on a generic circle of radius $\sqrt{n}$. Furthermore, Benatar, Marinucci and Wigman \cite{BMW} showed that $\mathcal{Q}(l,n, n^{1/2-\epsilon})=\emptyset$ for almost all $n\in S$.  Theorem \ref{theorem 1} not only refines their bound, but also gives an explicit dependence of $K$ on $l$ which is essential in the study of toral eigenfunctions at small scales.

The proof of Theorem \ref{thm 3} relies on the quantitative convergence of the spectral measure $\mu_n$ to $\mu$, given by Erd\"{o}s-Hall \cite{EH} and K\'{a}tai- K\"{o}rnyei \cite{KK}, see also Theorem \ref{Erdos-Hall} below, and a recent result of Beliaev-Maffucci \cite{BM}. In this scenario, the Continuous Mapping Theorem suggests that convergence of the spectral measure, in the weak sense, implies convergence in distribution  of the nodal length. This principle has already been rigorously implemented by Granville and Wigman \cite{GW2} for trigonometric polynomials and by Todino \cite{To} for spherical harmonics in the two dimensional sphere. The author was also recently communicated that Dierickx, Nourdin, Peccati and Rossi \cite{DNPR} showed that the said principle applies in a quite general scenario: They showed, from the appropriate convergence of covariances, the convergence to $\mathcal{L}(f_{\mu})$, in mean square and distribution, of the nodal length of Gaussian  monochromatic random waves on Riemann surfaces without conjugate points. As a consequence of their method, with the notation of Theorem \ref{thm 3}, they found the variance and the distribution of $\mathcal{L}(F_n)$ to agree with $\mathcal{L}(F_{\mu})$ for $B= \frac{1}{18}\log (\pi/2)$; thus giving a sharper, value for $B$ in our Theorem \ref{thm 3}. 

Finally, in light of the techniques in \cite{GW2}, it is conceivable that the value of $B$ can be increased further, to maybe $B=\log(\pi/2)/2=0.2257...$, using Crofton's formula and Hurwitz's theorem. Nevertheless, since this new value for $B$ would still be far from our expected $A_0$, we opted for a short and, in our view, elegant proof based on the stability of the nodal set, as in \cite{NS}, and a quantitative version of the Continuous Mapping Theorem. 
\subsection{Notation}
Let $u \rightarrow \infty$ be some parameter, we say that the quantity $X=X(u)$ and $Y=Y(u)$   satisfy $X\ll Y$ , $X\gg Y$ if there exists some constant $C$, independent of $u$, such that $X\leq C Y$ and $X\geq CY$ respectively. If $X\ll Y$ and $Y\ll X$, we write $X \asymp Y$. We also write $O(X)$ for some quantity bounded in absolute value by a constant times $X$ and $X=o(Y)$ if $X/Y\rightarrow 0$ as $u\rightarrow \infty$, in particular we denote by $o(1)$ any function that tends to $0$ (arbitrarily slowly) as $u\rightarrow \infty$. We denote by $B(s)$  the (open) ball centred at $0$ and by $\overline{B}(s)$ the closure of $B(s)$. When the specific radius is unimportant, we simply write the ball as $B$ and $\frac{1}{2}B$ for the concentric ball with half the radius. Moreover, for a positive integer $n$, we denote by $\omega(n)$ the number of its prime factors without multiplicity. Furthermore, given some $k\geq 0$ and some $k$-times differentiable function $f:B\rightarrow \mathbb{R}$, we denote by $||f||_{C^k(B)}= \sum_{m=0}^k \sup_{x\in B} |f^{(m)}(x)|$, where $f^(m)(x)$ is the $m$-th derivative.  Finally, we denote by $\Omega$ an abstract probability space where every random object is defined.

\section{Number theoretic background}
\subsection{An equivalent formulation of Theorem \ref{theorem 1}}
\label{NT background}
Given $n\in S$, we can express the representations $\xi_i$ of $n$ as products of prime ideal in $\mathbb{Z}[i]$. Let $n=2^{\alpha_2}\prod_{k} p_k^{\alpha_k}\prod_{v}q_v^{\beta_v}$, where $p_k$ and $q_s$ are primes $p\equiv 1 \pmod 4$ and $q\equiv 3 \pmod 4$ and $\alpha$'s and $\beta$'s are positive integers. Let $(n)\subset \mathbb{Z}[i]$ be the ideal generated by $n$, then, by unique factorisation of ideals in $\mathbb{Z}[i]$, bearing in mind that primes $p\equiv 1 \pmod 4$ split and primes $q\equiv 3 \pmod 4$ are inert, we have 
\begin{align}
(n)= Z^{2\alpha_2}\prod_{k} (\mathcal{P}_k\overline{\mathcal{P}_k})^{\alpha_k}\prod_{v} \mathcal{Q}_v^{\beta_v} \nonumber
\end{align}
where $Z$ is an ideal above $2$,  $\mathcal{P}_k$ an ideal above $p_k$ and  $\mathcal{Q}_v$ an ideal above $q_v$. Thus, if $(n)=(x+iy)(x-iy)$ for some $x,y \in \mathbb{Z}$, then 
\begin{align}
(x+iy)= Z^{\alpha_2}\prod_k\mathcal{P}_k^{\gamma_{kj}}\overline{\mathcal{P}_k}^{\alpha_k-\gamma_{kj}}\prod_{v} \mathcal{Q}_v^{\beta_v/2}.\label{7}
\end{align} 
for some $0\leq \gamma_{kj}\leq \alpha_k$. It follows that  the $\beta$'s must be even and  the representations $\xi_i$ of $n$ are in one to one correspondence with ideal of the form \eqref{7}.  Therefore, taking into account the symmetries $\xi\rightarrow -\xi$ and $x+iy \rightarrow y+ix$,  we have
\begin{align}
N(n)= 4\prod_{k}(\alpha_k+1)\ll \exp \left( O\left( \frac{\log  n}{\log\log n}\right)\right) \label{divisor bound}
\end{align}
where the inequality follows from the divisor bound. Moreover, by \eqref{7}, we see that the factor $Z^{\alpha_2}\prod_{v}\mathcal{Q}_v^{\beta_v/2}$ is common to every representation. This has the effect of rotating the lattice points, but it does not affect their spacial distribution. Thus, it does not effect the set of solutions to \eqref{4}. Hence, we can restate Theorem \ref{theorem 1} as follows
\begin{thm}
	\label{theorem 3}
	Let $S':=\{n\in S: p|n \hspace{2mm} \text{then} \hspace{2mm} p\equiv 1 \pmod 4 \}$, $l \geq 2$ be an integer, $\epsilon>0$ and $c(l,\epsilon)=c(l)$ be as in Theorem \ref{theorem 1}. Then there exists a density one subset of $n\in S'$ such that
$$
\mathcal{Q}(l,n, n^{1/2} / (\log n)^{c(l)})= \emptyset. 
$$
\end{thm}

\subsection{Spectral correlations}
\label{Spectral correlations}
Recall that $\mathcal{S}(l,n)=\{(\xi_1,...\xi_l): \xi_1+...+\xi_l=0\}$ and that, by congruence obstruction modulo $2$, $\mathcal{S}(l,n)=\emptyset$ if $l$ is odd. If $l$ is even, we have the \textquotedblleft diagonal" solutions given by $\xi_1=-\xi_2$,..., $\xi_{l-1}=-\xi_l$, thus  $\mathcal{S}(l,n)\gg N^{l/2}$. For $l=2$  the only solutions are $\xi_1=-\xi_2$, thus $\mathcal{S}(2,n)=N$. For $l=4$ Zygmund \cite{Z} observed that the only solutions are $\xi_1=-\xi_2$ and $\xi_3=-\xi_4$, therefore
\begin{align}
&\mathcal{S}(4,n)= 3N^2 + O(N) & N \rightarrow \infty. \nonumber
\end{align}  
For $l=6$ Bourgain \cite[Theorem 2.2]{KKW} showed that $\mathcal{S}(6,n)=o(N^4)$. Subsequently, Bombieri and Bourgain \cite{BB} gave the bound 
\begin{align}
\mathcal{S}(6,n)\ll N^{7/2}. \label{2}
\end{align}
Finally,  using the deep work of  Evertse-Schlickewei-Schmidt \cite{ESS} on additive relations in multiplicative subgroups of $\mathbb{C}^{\star}$ of bounded rank, see \cite[Theorem 5]{BB} and \cite[Lemma 5]{BU}, we can explicitly construct sub-sequences of $n\in S$ which satisfy \eqref{3}:
\begin{lem}
	\label{heavy}
	Let $n=\prod_{i}^rp_i^{\alpha_i}\in S'$  and $l\geq 2$ be an even integer. If $\sum_i \log (\alpha_i+1)/r \rightarrow \infty$, then the number of solutions to \eqref{3} is 
	
	\begin{align}
	\frac{l!}{2^{(l/2)} \cdot (l/2)!}N^{l/2} + O(N^{\gamma l}) \nonumber
	\end{align}
	for some $0<\gamma<1/2$. 
\end{lem}
\subsection{Lattice points and geometry of numbers}
 In this section we collect some facts which will be used thorough the rest of the article. By  Landau's Theorem, see for example \cite[Theorem 14.2]{FI}, there exists some explicit constant $c>0$ such that
\begin{align}
\#\{n\in S': n\leq X\}= c \frac{X}{\sqrt{\log X}}(1+o(1)) \label{landau}
\end{align}

Thanks to (a weak version of) the Erd\"{o}s-Kac Theorem, see for example \cite[Part III Chapter 3]{T}, we have 
\begin{lem}[Erd\"{o}s-Kac] 
	\label{ERKA}
	Let $\epsilon>0$, then,	for a density one subset of $n\in S'$, we have 
	\begin{align}
	\frac{1}{2}\log\log n (1 -\epsilon )\leq	\#\{p|n:p\equiv 1 \pmod 4\}\leq  \frac{1}{2} \log\log n (1 +\epsilon ). \nonumber
	\end{align}
	where the primes are counted without multiplicity. In particular, via \eqref{divisor bound}, we have 
		\begin{align}
	N(n)=N \asymp (\log n)^{\frac{\log 2}{2}\pm\epsilon}. \nonumber
	\end{align}
\end{lem}

As another consequence of (a slightly stronger version of) the Erd\"{o}s-Kac Theorem, we can also control the size of the prime in the factorisation of a generic integer $n\in S'$, see for example  \cite[Part III Chapter 3, Theorem 8 and Theorem 9]{T} for a standard derivation of the following fact from the Erd\"{o}s-Kac Theorem.
\begin{lem}
		\label{size of p_r}
	Let $n\in S'$ and let $p_1<p_2<...<p_r$ be its prime factors. Then for a density one subset of $n\in S'$ we have 
	\begin{align}
	\sup_{\log\log\log n<k< r} \left|\frac{2^{-1}\log\log p_k- k}{\sqrt{k\log k}}\right|\leq 3/2 .\nonumber
	\end{align}
	In particular, by Lemma \ref{ERKA}, we have 
	\begin{align}
	p_r \geq \exp( (\log n)^{1/3}). \nonumber
	\end{align}
\end{lem} 
\begin{rem}
	Lemma \ref{size of p_r} is not sharp, in particular the constant $1/3$ can be replace with a larger constant, but it will suffice for our purposes. 
\end{rem}

We will also need the following result of Kubilius \cite{K} about Gaussian primes, which are primes $\mathcal{P} \subset \mathbb{Z}[i]$ such that $|\mathcal{P}|^2=p$ with $p\equiv 1 \pmod 4$. 
\begin{lem}[Kubilius]
	\label{Kubilius}
	Let $\theta_1,\theta_2\in [0,2\pi]$.  Then, the number of Gaussian primes in the sector $\arg(\mathcal{P})\in [\theta_1,\theta_2]$ such that $|\mathcal{P}|^2\leq X$ is 
	\begin{align}
	\frac{2}{\pi}(\theta_1-\theta_2)\int_{2}^{X}\frac{dx}{\log x} + O (X\exp(-c\sqrt{\log X})). \nonumber
	\end{align}
\end{lem}

Finally, we will need the following result about the distribution of lattice points for generic $n$.  Recall the spectral measure $\mu_n$ in \eqref{measure} and $\mu$, the Lebesgue measure on the interval $[0,1]$, then have the following theorem, see \cite{EH,KK}.
\begin{thm}[Erd\"{o}s-Hall]
	\label{Erdos-Hall}
	Let $\kappa= \frac{1}{2}\log \frac{\pi}{2}$ and $\epsilon>0$. Then, for a density one subset of $n\in  S$, we have 
	\begin{align}
	\sup_{0<a<b<1}|\mu_n(a,b)- \mu(a,b)|\leq (\log n)^{-\kappa+\epsilon}. \nonumber
	\end{align}
\end{thm}
\section{Proof of Theorem \ref{theorem 1}}
The argument in this section is inspired by \cite[Theorem 14]{BB}. As discussed in section \ref{NT background} it is enough to prove Theorem \ref{theorem 3}. To ease the exposition we divide the proof into two parts: $n$ square-free and $n$ not square-free. We begin by proving Theorem \ref{theorem 3} in the square-free case.
\subsection{Proof of Theorem \ref{theorem 3} for the square-free case}
During the proof of Theorem \ref{theorem 3} we will need the following direct consequence of the Landau's Theorem \eqref{landau}: 
\begin{lem}
	\label{contazzi}
	Let $X>1$ be some large parameter and $S'$ be as in section \ref{NT background}, then we have 
	$$\sum_{\substack{n \in S' \\ n \leq X/100}} \frac{1}{n\log(X/n)}\ll (\log X)^{-1/2}.$$ 
\end{lem}
\begin{proof}
	Using the expansion $\log(X/n)^{-1}= (\log X)^{-1}(1 + O(\log n/\log X))$ valid for $n\leq X/100$ say,  we have
	\begin{align}
	\sum_{\substack{n\in S' \\ n\leq X/100}}  \frac{1}{n\log (X/n)} \ll \frac{1}{\log X}\sum_{\substack{n\in S' \\ n\leq X/100}}\frac{1}{n}+ \frac{1}{(\log X)^2}\sum_{\substack{n\in S' \\ n\leq X/100}}\frac{\log n}{n}   \label{14.1}
	\end{align}
	By partial summation using \eqref{landau}, we have 
	\begin{align}
	&\sum_{\substack{n\in S' \\ n\leq X/100}} \frac{1}{n}\ll (\log X)^{1/2} & \sum_{\substack{n\in S' \\ n\leq X/100}} \frac{\log n}{n}\ll (\log X)^{3/2}  \label{15}
	\end{align}
	Thus, the lemma follows by inserting \eqref{15} into \eqref{14.1}.  
\end{proof}
Let $X>1$ be some (large) parameter, for a positive integer $n$, which we assume to be always square-free, let $\omega(n)$ be the number of its prime divisors without multiplicity.  Moreover, given some integer $l\geq 2$ and $\epsilon>0$, we define \begin{align}
	\Phi(n,l)=\Phi(n)= (\log n)^{c(l)} \label{defPhi}
	\end{align}
where $c(2,\epsilon)=c(2)=\log 2 +\epsilon$, $c(3,\epsilon)=c(3)=3\log 2/2+\epsilon$ and $c(l,\epsilon)=c(l)=l\log 2/2+c(\lfloor l/2 \rfloor) + \epsilon$ and  $\lfloor l/2 \rfloor $ represent the largest integer smaller than $l/2$.  Finally, let $\mathcal{F}_r(X,l,\epsilon)=\mathcal{F}_r(l)$  be the set of $n\in S'$ such that $\omega(n)=r$ and 
$$\mathcal{Q}(l,n,n^{1/2}/\Phi(n))\neq \emptyset,$$ but for all divisors $d|n$ and $d\neq n$, we have 
$$\mathcal{Q}(l,d,d^{1/2}/\Phi(d))= \emptyset.$$

We are going to prove following bound:
\begin{prop}
	\label{main prop}
	Let $l\geq 2$ be  a positive integer and $\epsilon>0$, then for a density one subset of $n\in S'$, we have 
	$$|\mathcal{F}_r(X,l,\epsilon)|\ll_{l}   \frac{X}{(\log X)^{1/2  +\epsilon/2}}, $$
	uniformly for all $r\geq 1$.  
\end{prop}

\begin{proof}
First, by Lemma \ref{ERKA} and Lemma \ref{size of p_r}, we may assume that for all $\epsilon>0$ and all $n\in F_r(l)$ the following hold:
\begin{align}
	\label{assumptions}
	&r\leq \frac{1}{2} \log\log X(1+ \epsilon/4) & p_r \gg \exp((\log X)^{1/3}) 
\end{align}
where $p_r$ is the largest prime divisor of $n$. We are now going to prove the Proposition by induction on $l$.

\textbf{Base case, $l=2$.}
Let $n\in \mathcal{F}_r(2)$, then there exist  two lattice points $\xi_1\neq \xi_2$ such that $|\xi_1|^2=|\xi_2|^2=n$ and 

	\begin{align}
0<\left|\xi_1-\xi_2\right|< n^{1/2}/\Phi(n). \label{8}
\end{align}

	Let $p_1<p_2<...<p_r$ be the prime factors of $n$. With the same notation as in \eqref{7} and bearing in mind that $n$ is square-free, we may write 
	\begin{align}
	\xi_i= \prod_{k=1}^r \overline{\mathcal{P}_k}^{\gamma_{ki}}\mathcal{P}_k^{1-\gamma_{ki}} \nonumber
	\end{align}
	where $\gamma_{ki}\in \{0,1\}$ and $\mathcal{P}_k$ is the Gaussian prime above $p_k$ (which we select by insisting that $\Arg( \mathcal{P}_k) \in (0,\pi/2)$). Fixing $\mathcal{P}_1,...\mathcal{P}_{r-1}$, we can rewrite \eqref{8}, in view of the fact that $n\in F_r(2)$, as 
	\begin{align}
	0< | \mathcal{P}_r a- \overline{\mathcal{P}_r} b|\leq n^{1/2}/ \Phi(n) \label{9}
	\end{align} 
	for some $a,b \in \mathbb{Z}[i]$ such that $|a|=|b|= (n/p_r)^{1/2}$. Thus, dividing both sides of \eqref{9} by $n^{1/2}$, we deduce
	\begin{align}
	\Arg(\mathcal{P}_r)\in \left[ \theta- \frac{1}{\Phi(n)}, \theta+ \frac{1}{\Phi(n)}\right]:=I_{a,b}(n)=I(n) \label{10}. 
	\end{align}
	for some $\theta=\theta(a,b)$. Thus, given  $p_1<...<p_{r-1}$, there are at most $2^{2(r-1)}$ choices of for $a$ and $b$, once these are fixed we obtain that $p_r$ is a prime of modulus $\exp ((\log X)^{1/3})\ll p_r \leq X/ \prod_{k}^{r-1}p_k$ and the argument of the Gaussian prime $\mathcal{P}_r$ above it satisfies \eqref{10}. Hence, we have the following key bound
	\begin{align}
	\#\{n\in \mathcal{F}_r(l): n\leq X \hspace{2mm}\text{and} \hspace{2mm} \mathcal{Q}(l,n,n^{1/2}/\Phi(n))\neq \emptyset\}\leq \sum_{p_1<...<p_{r-1}}2^{2 r} \sum_{\substack{|\mathcal{P}_r|^2\leq X/ \prod_{k}^{r-1}p_k \\ |\mathcal{P}_r|^2 \gg \exp ((\log X)^{1/3})\\ 	\Arg(\mathcal{P}_r) \in I(n) }}1. \label{11} 
	\end{align}
		 Since $\Phi(n)\asymp \Phi(X)$, letting $Y=X/\prod_{k}^{r-1}p_k$ and bearing in mind that $\int_2^Y (\log x)^{-1}dx= Y\left((\log Y)^{-1}+ O\left((\log Y)^{-2}\right)\right)$, Lemma \ref{Kubilius}  gives
	\begin{align}
	\sum_{\substack{|\mathcal{P}_r|^2\leq X/ \prod_{k}^{r-1}p_k \\ 	\Arg(\mathcal{P}_r) \in I(n) }} 1\ll \frac{1}{\Phi(X)} \cdot \frac{Y}{\log Y} + Y \exp(-c\sqrt{\log Y}) .  \label{12}
	\end{align}

We first consider the first term on the right hand side of \eqref{12}: let $Y= |\mathcal{P}_r|^2\gg \exp((\log X)^{1/3}):=Z$, then
	\begin{align}
\sum_{p_1<...<p_{r-1}} \frac{Y}{\log Y} \ll X \sum_{\substack{p_1<p_2<...<p_{r-1} \\ \prod_{k}p_k \leq X/Z}} \frac{1}{\prod_{k}p_k \log (X/\prod_{k}p_k)}\leq \sum_{\substack{n\in S' \\ n\leq X/Z}} \frac{1}{n\log (X/n)} \label{13}
	\end{align}
	where we have extended the inner sum by lifting the restriction on the number of prime factors of $n$. 
	Hence, combining Lemma \ref{contazzi} and \eqref{13}, we have 
	\begin{align}
\frac{1}{\Phi(X)}\sum_{ p_1,...,p_{r-1}} \frac{Y}{\log Y}\ll \frac{X}{(\log X)^{1/2}\Phi(X)}  \label{main term}
	\end{align}

We now consider the second term on the right hand side of \eqref{12}: 
	\begin{align}
\sum_{p_1,...,p_{r-1}} Y \exp(-c\sqrt{\log Y})&\leq X \sum_{\substack{\prod_{k}p_k \leq X/Z}} \frac{\exp(-c\log (X/\prod_{k}p_k)^{1/2})}{\prod_{k} p_k} \nonumber \\
	&\overset{X/\prod_{k}p_k \gg Z}{\leq} X  \exp(-c\log (Z)^{1/2})\sum_{\substack{n \in S' \\ n \leq X}} \frac{1}{n}  \nonumber \\
	&\overset{\eqref{15}}{\ll} X  \exp(-c\log (X)^{1/6})   \label{error term}
	\end{align}
	which is smaller than \eqref{13}.  Hence the base case follows from inserting \eqref{12} into \eqref{11}, using \eqref{main term} and \eqref{error term} and noticing that \eqref{assumptions} gives $2^{2r}/\Phi(X)\ll (\log X)^{-\epsilon/2} $. 

\textbf{Induction step.}
Let us assume that the Proposition holds for all $l'<l$, then, repeating the argument in the base case and maintaining the same notation, we find the following: given $n\in \mathcal{F}_r(l)$ with prime factorisation $p_1<...<p_r$, let $\mathcal{P}_i$ be the Gaussian prime above $p_i$, then 
	\begin{align}
	0< | \mathcal{P}_r a- \overline{\mathcal{P}_r} b|\leq n^{1/2}/ \Phi(n,l) \label{9.1}
\end{align} 
for some non-zero $a,b \in \mathbb{Z}[i]$ depending $\mathcal{P}_1,...,\mathcal{P}_{r-1}$. Now, we claim that for all but at most $O(X/(\log X)^{1/2+\epsilon/2})$ integers $n\in S'$ up to $X$, we have 
\begin{align}
	|a| \geq \left(\frac{n}{p_r}\right)^{1/2} \cdot \frac{1}{\Phi(n/p_r,\lfloor l/2\rfloor)} \hspace{5mm}\text{or} \hspace{5mm} 	|b| \geq \left(\frac{n}{p_r}\right)^{1/2} \cdot \frac{1}{\Phi(n/p_r,\lfloor l/2\rfloor)} \label{9.3}
\end{align}
Indeed, since either $a$ or $b$ is the sum of at most $\lfloor l/2 \rfloor$ lattice points on the circle of radius $(n/p_r)^{1/2}$,  the induction hypothesis, with $2\epsilon$ instead of $\epsilon$,  gives that the number of exceptions is at most 
\begin{align}
	\sum_{p_r\leq X/2^r} \mathcal{F}_{r-1}\left( \frac{X}{p_r}, \lfloor l/2\rfloor\right)&\ll X \sum_{ p\leq X/2} \frac{1}{p(\log X/p)^{1/2+\epsilon}} \nonumber \\
	&\ll_{\epsilon} \frac{X}{(\log X)^{1/2+\epsilon}}\left( \sum_{p\leq X} \frac{1}{p} + \frac{1}{\log X} \sum_{p\leq X}\frac{\log p}{p}\right) \nonumber \\
	&\ll \frac{X \log \log X}{(\log X)^{1/2+\epsilon}} \ll  \frac{X}{(\log X)^{1/2+\epsilon/2}} \nonumber
	\end{align}
where, in the second but last line, we have used the expansion $(\log(X/p))^{1/2+\epsilon}= \log X^{1/2+\epsilon}(1+O_{\epsilon}(\log p/\log X))$. Thus, we have proved the claim and we may assume that \eqref{9.3} holds. 

Without loss of generality we might assume that $|a|\geq |b|$, so that, inserting \eqref{9.3} into \eqref{9.1}, we obtain 
	\begin{align}
	0< \left| \mathcal{P}_r - \overline{\mathcal{P}_r} \frac{a}{b}\right|\leq p_r^{1/2} \frac{\Phi (n/p_r,\lfloor l/2 \rfloor)}{ \Phi(n,l)}\leq   p_r^{1/2} \frac{\Phi (n,\lfloor l/2 \rfloor)}{ \Phi(n,l)}=:  p_r^{1/2}\Psi(n).  \nonumber
\end{align} 
Following step by step the computations in the base case with $\Phi$ substituted by $\Psi$, bearing in mind the number of exceptions in claim \eqref{9.3}, we deduce that 
$$ \mathcal{F}_r(l)\ll 2^{l r} \frac{X}{(\log X)^{1/2}}  \cdot \frac{1}{\Psi(X)} + O\left( \frac{X}{(\log X)^{1/2+\epsilon/2}}\right) $$ 
which, in light of \eqref{defPhi} and \eqref{assumptions},  concludes the induction. 
\end{proof}
 We are finally ready to prove Theorem \ref{theorem 3}:
 \begin{proof}[Proof of Theorem \ref{theorem 3} for the square-free case]
 	Let $l\geq 2$ and $X>1$ be given, if $n\in S'$ less then $X$ is such that $\mathcal{Q}(l,n,n^{1/2}/\Phi(n))\neq \emptyset$ then $n\in F_r(l)$ for some $r\geq 1$. Moreover, by Lemma \ref{ERKA}, we may assume that $r\leq \log\log X$. Hence, summing the bound in Proposition \ref{main prop} over the all possible values of $r$, we obtain the Theorem.  
 \end{proof}
\subsection{Proof of Theorem \ref{theorem 1} for the non square-free case}
We are now going to prove Theorem \ref{theorem 3} in the non-square free case. We need the following standard lemma: 
\begin{lem}
	\label{comparison}
	Let $n\in S'$ and let $\omega(n)$ be the number of prime factors of $n$ without multiplicity and $\Omega(n)$ be the number of prime factors of $n$ with multiplicity, then 
	\begin{align}
	\sum_{\substack{n\leq X \\ n\in S'}} |\Omega(n)-\omega(n)|\ll \frac{X}{\sqrt{\log X}} \nonumber
	\end{align}
\end{lem}
\begin{proof}
Using Landau's Theorem \eqref{landau} and following similar calculations as in the proof  of Proposition \ref{main prop}, we obtain
	\begin{align}
	\sum_{\substack{n\leq X \\ n\in S'}} |\Omega(n)-\omega(n)|\leq X\sum_{\substack{p \hspace{1mm}\text{prime} \\ p \equiv 1 \pmod 4}}\sum_{i\geq 2} \frac{1}{p^i (\log(X/p^i))^{1/2}} \ll \frac{X}{\sqrt{\log X}} \sum_p \frac{1}{p^2}\ll \frac{X}{\sqrt{\log X}} \nonumber
	\end{align}
\end{proof}
We are now ready to begin the proof of Theorem \ref{theorem 3} in the non-square free case. 
\begin{proof}[Proof of Theorem \ref{theorem 3} for $n$ not square-free] 
	First, let us define  $\tilde{\mathcal{F}}_r(l)$ to be $\mathcal{F}_r(l)$ but  $\omega(n)$ is substituted by $\Omega(n)$. Observe that, by Lemma \ref{comparison} and with the same notation, we have, for a density of subset of $n\in S'$, $\Omega(n)-\omega(n)\leq \log\log\log n$. Thus, we may assume that $r\leq 2\log \log n$. Thus, as we have seen in the proof of Theorem \ref{theorem 3} for the square-free case, it is enough to show that 
	\begin{align}
		\tilde{\mathcal{F}}_r(l) \ll \frac{X}{\log (X)^{1/2+\epsilon/2}}. \label{claim square-free}
		\end{align}
	
	To prove \eqref{claim square-free},  we may assume, by again Lemma \ref{comparison}, that at most $O(\log\log\log n)$ of the prime divisors of $n$ have multiplicity. Therefore, given $n\in S'$ and letting $p_1<...<p_r$ be its prime factors, by Corollary \ref{size of p_r} we may assume that there exists some $s\geq \log\log n/10$ such that 
	\begin{align}
	p_s \gg \exp( (\log X)^{1/3}) \nonumber 
	\end{align} 
	and moreover $(p_s)^2$ does not divide $n$. Fixing $s$, the proof now proceeds step by step as the proof of Proposition \ref{main prop} in the square-free case. Finally, summing the bound in Proposition \ref{main prop} over the $\log\log n$ choices  for $s$ gives \eqref{claim square-free}, up to changing the values of $\epsilon$. 
\end{proof}
\section{Random model for lattice points, proof of Theorem \ref{thm random}}
\label{Random model}
As discussed in section \ref{NT background}, we may assume that a generic integer $n\in S'$ has $\omega(n)\asymp \log\log n$ prime factors, most of which are not repeated by Lemma \ref{comparison}, and the distribution of the angle of Gaussian prime is uniform in $[0,2\pi)$. Thus, we may model representations of a generic integer $n\in S'$ as 
\begin{align}
\xi_i= \exp\left( 2\pi i \sum_{k=1}^{\omega(n)}\eta_{ik}\theta_k\right) \label{model}
\end{align}
where $\theta_k$'s are i.i.d random variables uniformly distributed on $[0,1)$ and $\eta_{ik}\in \{-1,1\}$ are deterministic, see also \cite[Remark 3.3]{GW}. This gives $2^{\omega(n)}\asymp N$ representations of $n$. Given $l\geq 2$, we define the random variables 
\begin{align}
X_{\underline{i}}= \xi_{i_1}+...+\xi_{i_l}\nonumber
\end{align}
for $1\leq i_j\leq N$. Before proving Theorem \ref{thm random}, we need two preliminary results. 
\subsection{Distribution and independence of the $\xi_i$}
In this section we show that the random $\xi_i$ as in \eqref{model} have the same distribution and are \textquotedblleft generically \textquotedblright independent.
\begin{lem}
	\label{distribution}
	Let $\xi_i$ be as in \eqref{model} and let $r\in \mathbb{Z}$ be some integer, then 
	\begin{align}
	\mathbb{E}[ \xi_i^r] = \begin{cases}
	1 & r=0 \\
	0 & r\neq 0 
	\end{cases} .\nonumber
	\end{align}
\end{lem}
\begin{proof}
	Since the $\theta_k$ are independent, we have 
	\begin{align}
	\mathbb{E}[\xi_i^r]= \prod_{k}\int_0^1 e\left(r \eta_{jk}\theta_k\right) d\theta_k \nonumber
	\end{align}
	which gives the lemma as $\int  e\left(r \eta_{jk}\theta_k\right) d\theta_k=0$ unless $r=0$. 
\end{proof}
\begin{lem}
	\label{independence}
	Let $l \geq 2$ and let $\xi_1,...\xi_l$ be as in \eqref{model}. Then $\xi_1,...\xi_l$ are independent for all but $o_{\omega(n)\rightarrow \infty}(2^{\omega(n)l})$ choices of $l$-tuples $(\xi_1,...,\xi_l)$. 
\end{lem}
\begin{proof}
	Since $|\xi_i|\leq 1$ surely, the joint distribution of $(\xi_1,...,\xi_l)$ is fully determinated by its moments. Thus, it is enough to prove that given $l$ integers $m_1,...m_l$, we have 
	\begin{align}
	\mathbb{E}\left[ \prod_{i=1}^{l}\xi_i^{m_i}\right]= \begin{cases}
	1 &m_1=m_2=...=m_l=0 \\
	0 & \text{otherwise}.
	\end{cases}
	\end{align}
	Observe that 
	\begin{align}
	\mathbb{E}\left[\prod_{i=1}^{l} \xi_i^{m_i}\right]= \mathbb{E}\left[ \exp\left(2\pi i \sum_k (\sum_{i=1}^l m_i\eta_{ik})\phi_k \right)\right],
	\end{align}
	integrating, we obtain that $m_1 \eta_{1k}+...+ m_l\eta_{lk}=0$ for all $k$. Therefore $(m_1,...,m_l)$ is in the kernel of the matrix $\{\eta_{ik}\}$ for $1\leq i \leq l$ and $1\leq k \leq \omega(n)$. Hence, it is enough to prove the following claim:
	\begin{claim}
		Consider the $l\times\omega(n)$ random matrix with entry $\eta_{ik}=1$ with probability $1/2$ and $\eta_{ik}=-1$ with probability $1/2$. Then with probability greater than $1-o_{\omega(n)\rightarrow \infty}(1)$, the matrix $\{\eta_{ik}\}$ for $1\leq i\leq l$ and $ 1\leq k \leq \omega(n)$ has rank $l$. 
	\end{claim} 
	\begin{proof}
		Consider the first two rows $\{\eta_{1k}\}$ and $\eta_{2k}$, outside a set $\Omega_1$ of probability at most $2 \cdot 2^{-\omega(n)}$ there exists $k_1$ and $k_2$ such that $\eta_{1k_1}=\eta_{2k_1}$ and $\eta_{1k_2}\neq \eta_{2k_2}$. Therefore, applying row and columns operations, we reduce the matrix as 
		\begin{align}
		&\begin{bmatrix}
		1& 1 &... \\
		-1 & 1 &... \\
		... \\
		... 
		\end{bmatrix} \hspace{20mm}\rightarrow 	&\begin{bmatrix}
		1& 1 &... \\
		0 & 2 &... \\
		0& 0 &... \\
		0 &0 &...
		\end{bmatrix}. \nonumber
		\end{align}
		Consider now rows $\{\eta_{3k}\}$ and $\{\eta_{4k}\}$. By the above, we have $\eta_{3k}=\eta_{4k}=0$ for $k=1,2$, and for $k\geq 3$ the entries are $\{-2,0,2\}$ with probability $1/4$, $1/2$ and $1/4$ respectively. Therefore, outside a set $\Omega_2$ of probability at most $2\cdot 2^{-\omega(n)+2}$, there exists some $k_3,k_4\geq 3$ such that $\eta_{3k_3}= \eta_{4k_3}\neq 0$ and $\eta_{3k_4}\neq \eta_{4k_4}$.
		Therefore, applying row and columns operations, we have one of the following matrices 
		\begin{align}
		&\begin{bmatrix}
		1& 1 &... \\
		0 & 2 &... \\
		0& 0 & 2 & 2 &... \\
		0 &0 & 0 & 2&... \\
		... \\ 
		\end{bmatrix} \hspace{5mm}\text{or}&\begin{bmatrix}
		1& 1 &... \\
		0 & 2 &... \\
		0& 0 & 2 & 2 &... \\
		0 &0 & 2 & -2&... \\
		... \\
		\end{bmatrix} \hspace{10mm}\rightarrow \hspace{10mm}	&\begin{bmatrix}
		1& 1 &... \\
		0 & 2 &... \\
		0& 0 & 2 & 2 &... \\
		0 &0 & 0 & 4&... \\
		... \\
		\end{bmatrix}. \nonumber
		\end{align} 
		Now consider rows $\{\eta_{5k}\}$ and $\{\eta_{6k}\}$. By the above, we have $\eta_{5k}=\eta_{6k}=0$ for $k\leq 4$, and for $k\geq 5$ the entries are $\{-4,-2,0,2,4\}$ with probability $1/8$, $2/9$,$1/3$, $2/9$ and $1/8$ respectively. Since $l$ is fixed, we can repeat the above argument to find subsets $\Omega_1,...\Omega_{\lceil l/2\rceil}$ such that $\mathbb{P}(\Omega_i)=o(1)$ and outside $\cup_i \Omega_i$, we have $l$ linearly independent columns, as required.   
	\end{proof}
Having proved the claim, the Lemma follows. 
\end{proof}
\subsection{Proof of Theorem \ref{thm random}}
We are finally ready to prove Theorem \ref{thm random}. 
\begin{proof}[Proof of Theorem \ref{thm random}]
	Since $|\xi_i|\leq 1$ surely, we also have $|\re(\xi_i)|= |2^{-1}(\xi_i+\overline{\xi_i})|\leq 1$ and $|\im(\xi_i)|=|2^{-1}(\xi_i-\overline{\xi_i})|\leq 1$. Therefore, by Levi's continuity Theorem \cite[Theorem 26.2]{B}, the distribution of $\xi_i$ is fully determinated by its integer moments. Then Lemma \ref{distribution} implies that each $\xi_i$ have the same distribution and in particular they are uniformly distributed on the unit circle $S^1\subset \mathbb{R}^2$. Given $l\geq 2$ suppose that $\xi_{i_1},...,\xi_{i_l}$ are i.i.d. random variables and identify $S^1$ with the unit interval $[-1/2,1/2)$, then the random variable 
	\begin{align}
	X_{\underline{i}}= \xi_{i_1}+...+\xi_{i_l} \nonumber
	\end{align}
	has Irwin–Hall distribution. In particular, the density function of $X_{\underline{i}}$ is piece-wise analytic and depends only on $l$. Therefore, given any $\alpha>0$, by Taylor's expansion, we  have 
	\begin{align}
	\mathbb{P}(|X_i|\leq \alpha)\asymp_l \alpha (1+O(\alpha)). \label{1.2}
	\end{align}
 By Lemma \ref{independence} we have that $\xi_1,...\xi_l$ are independent for all but $o_{\omega(n)\rightarrow \infty}(2^{\omega(n)l})$ choices of $\xi_1,...\xi_l$. Thus, bearing in mind \eqref{1.2}, we obtain 
	\begin{align}
	\mathbb{E}[\#\{X_i: |X_i|\leq \alpha\}]\asymp_l 2^{\omega(n)l}(\alpha + O(\alpha^2))(1+o_{\omega(n)\rightarrow \infty}(1)) \nonumber
	\end{align}
	which concludes the proof. 
\end{proof}
\section{Proof of Theorem \ref{thm 2}}
The proof of Theorem \ref{thm 2} essentially follows the proof  of the main Theorem in \cite{BMW} and,  for the sake of completeness, we summarise here the main steps.  The main difference is that we explicitly construct sequences of $n\in S$ for which we can control the distribution of lattice points on  $\sqrt{n}S^1$. This is the content of the next section. 
\subsection{Limit points and spectral correlations}
In this section, we prove the following proposition: 
\begin{prop}
	\label{limit point}
	Let $w\in [0,1]$ and $l\geq 2$ be an integer. Then, there exists a sub-sequence of integers $n\in S'$ such that $N\rightarrow \infty$ as $n\rightarrow \infty$, $\hat\mu_n(4)\rightarrow w$,  \eqref{3} hold and $ \overline{\mathcal{Q}}(l,n, c(n)n^{1/2})= \emptyset$ for any function $c(n)\rightarrow 0$ arbitrarily slowly. 
\end{prop}
In order to prove Proposition \ref{limit point}, we need a two preliminary results. The first is Lemma \ref{heavy} in section \ref{NT background}. The second is a standard tool to control the size of exponential sums, see \cite{N}.
\begin{lem}[Remez' inequality]
	\label{NT}
	Let $F(t)=\sum_{i=1}^{J}a_{\eta}e(\eta_i \cdot t)$ where $t\in \mathbb{R}$, $J\in \mathbb{N}$, $a_{\eta}\in \mathbb{C}$ and suppose that $\eta_i \in \mathbb{R}$ are distinct.  Then, for any interval $B\subset \mathbb{R}$ and any sub-interval $I\subset B$, we have 
	\begin{align}
	\underset{ I}{\sup} |F|>\left(C\frac{|I|}{|B|}\right)^{J-1} 	\underset{B}{\sup} |F|. \nonumber
	\end{align}
	for some explicit $C>0$ independent of $F$.	 
\end{lem}
We are now ready to prove Proposition \ref{limit point}. 
\begin{proof}[Proof of Proposition \ref{limit point}]
	We pick integers of the form $n=p^m \cdot q$ for some primes $p,q\equiv 1 \pmod 4$ and $m\geq 1$ to be chosen later. Let $\theta_p,\theta_q$ be the angle of the Gaussian prime lying above $p$ and $q$ respectively. Observe that, by Lemma \ref{heavy} and \eqref{divisor bound},  if  $m\rightarrow \infty$ as $n\rightarrow \infty$, then  both \eqref{3} and $N\rightarrow \infty$ are satisfied. So we assume that $m$ is a sufficiently slow growing function of $n$ to be specified later. The rest of the proof relies on two claims. 
	\begin{claim}
		\label{claim 1}
		Let $\epsilon>0$. Then, there exist $\delta_1=\delta_1(\varepsilon)>0$ and some interval $I=I(\epsilon,m) \subset [0,2\pi)$ such that for all $\theta_q\in(0,\delta_1)$ and $\theta_p\in I$, we have
		\begin{align}
		|\hat{\mu}_n(4)-w| \leq \epsilon. \label{16}
		\end{align}
	\end{claim}
	\begin{proof}[Proof of Claim \ref{claim 1}]
		Define the de-symmetrize probability measure on $\mathbb{S}^1$ to be $d\nu_n(\theta)= d\mu_n (\theta/4)$. Then $\hat{\mu}_n(4)= \hat{\nu}_n(1)$ and, by convolution properties of the Fourier transform, we also have 
		\begin{align}
		\nu_n = \nu_{p^m} \star  \nu_{q}. \label{17}
		\end{align} 
		A direct computation shows that 
		\begin{align}
		&\hat{\nu}_{p^m}(1)= \frac{1}{m+1}\sum_{j=0}^m \cos ((m-2j)\theta_p) = \frac{\sin((m+1)\theta_p)}{(m+1)\sin \theta_p} &\hat{\nu}_{q}(1)= \frac{\sin(2\theta_q)}{2\sin\theta_q} \nonumber
		\end{align}
		Thus, using  \eqref{17} and properties of the Fourier Transform, we deduce
		\begin{align}
		\hat{\nu}_n(1)= \hat{\nu}_{p^m}(1)\cdot \hat{\nu}_{q}(1)=\frac{\sin((m+1)\theta_p)}{(m+1)\sin \theta_p}  \frac{\sin(2\theta_q)}{2\sin(\theta_q)}.\label{18}
		\end{align} 
		Observe that, since the function $\sin(2 x)/2\sin(x)$ tends to $1$ as $x\rightarrow 0$ and it is decreasing in a small neighbourhood to the right of $x=0$, we can find some small  $\delta_1=\delta_1(\epsilon)>0$  such that for all $x\in (0, \delta_1)$, we have  $|\sin(2x)/2\sin(x)- 1| \leq\epsilon/2$.   With this choice of $\delta_1$, bearing in mind that $|\sin(m\theta)/(m\sin \theta)|\leq 1$, equation \eqref{18} becomes
		\begin{align}
		\left|\hat{\nu}_n(1) - \frac{\sin((m+1)\theta_p)}{(m+1)\sin\theta_p}\right|\leq \epsilon/2  \label{19}
		\end{align}  
		The claim follows by \eqref{19} and the continuity of  $\sin((m+1)\theta)/((m+1)\sin \theta)$.
	\end{proof}
	Before stating the next claim, we introduce some notation. Observe that are only finitely many (depending on $m$ and $l$) sums $0\neq \sum_{i}^l \xi_i$, let us label them as $S_1$, $S_2$... . By \eqref{7},  we can write 
	\begin{align}
	\xi_i= \sqrt{n}e( a_i\theta_p \pm \theta_q + b_i\pi/2) \label{20}
	\end{align}
	for some $|a_i|\leq m$ and $b_i\in\{0,1,2,3\}$. Thus, collecting terms with equal $a_i$, we can write
	\begin{align}
	S_k= \sqrt{n}\sum_{j=-m}^m c_{jk} e(j \theta_p ):= \sqrt{n}\cdot F_k(\theta_p) \label{21}
	\end{align}
	for some $c_{ik}\in \mathbb{C}$ with $|c_{ik}|\leq 2l$.
	\begin{claim}
		\label{claim 2}
		Let $F_k(\theta)$ be as \eqref{21}. With the notation of Claim \ref{claim 1}, there exists some $q\asymp n^{1/m+1}$ with $\theta_q\in (0,\delta_1)$ and some $\delta_2= \delta_2(\epsilon,m)>0$ such that $\min_k \max_{I} |F_k(\theta)|\geq \delta_2$
	\end{claim}	
	\begin{proof}
		First, we prove that there exists some $q\asymp n^{1/m+1}$ with $\theta_q\in (0,\delta_1)$ such that  $\max_{[0,1]}|F_k(\theta)|\geq \delta_3$ for some $\delta_3=\delta_3(\epsilon,m,l)$ independent of $k$. Bounding the infinity norm by the $L^2$-norm, we have
		\begin{align}
		\max_{[0,1]}|F_k(\theta)|^2\geq \int |F_k(\theta)|^2 d\theta= \sum_i |c_{ik}|^2.  \label{22}
		\end{align}
		Since $F_k(\theta)$ is not identically zero, there exists some $j_1=j_1(k)$ such that $|c_{j_1k}|=|c_{j_1}|>0$.  By \eqref{20}, we have 
		\begin{align}
		c_{j_1}= \alpha e(\theta_q) + \beta e( -\theta_q) \nonumber
		\end{align}
		for some $\alpha=\alpha(k),\beta=\beta(k) \in \mathbb{Z}[i]$ with $|\alpha|\leq l$ and $|\beta|\leq l$. If $|\alpha|\neq |\beta|$, then $|c_{j_1}|\geq ||\alpha|-|\beta||\geq 1$. If $|\alpha|=|\beta|$, then, writing $\alpha=re(\psi_1)$ and $\beta=re(\psi_2)$, we have  
	$$ |c_{j_1}|= r|e(2\theta_q +\psi_1-\psi_2)+1|=\sqrt{2}r|1+\cos(2\pi(2\theta_q +\psi_1-\psi_2))|. $$
	Since there are at most $O(l^2)$ choices for $\psi_1$ and  at most $O(l^2)$ for $\psi_2$, there exists some $\theta_0=\theta_0(l) \in (0,\delta_1)$ such that 
$$|c_{j_1}| \geq \delta_3$$
for some $\delta_3=\delta_3(\varepsilon,m,l)>0$ independent of $k$. Via Lemma \ref{Kubilius}, we choose $q \asymp n^{1/m+1}$ such that $\theta_q= \theta_0 +O(\log n^{-100})$. Thus, bearing in mind that the continuity of $F_k$ depends only on $m$ and $l$,  together with \eqref{22}, we have $\max_{[0,1]}|F_k(\theta)|\geq \delta_4$, for some $\delta_4=\delta_4(\epsilon,m,l)$. Therefore, applying Lemma \ref{NT} with $B=[0,1]$ and $\Omega= I$, we deduce that
		\begin{align}
		\max_{I} |F_k(\theta)|\geq C^{-m}|I|^{-m}\delta_4:=\delta_2 \nonumber	
		\end{align} 
		uniformly for every $k$ and for some absolute constant $C>0$, as required. 
	\end{proof}
	
	To conclude the proof, we need to choose a sequence of primes $p,q\equiv 1 \pmod 4$ and a function $m=m(n)$ such that $\hat{\mu_n}(4)\rightarrow \omega$ and $\mathcal{Q}(l,n,c(n)n^{1/2})=\emptyset$. Let $\epsilon>0$ and $I$ and $q$ be given by Claim \ref{claim 1} and Claim \ref{claim 2} respectively. Via Claim \ref{claim 2}, let $\theta_1\in I$ be such that $\min_k |F_k(\theta_1)|>\delta_2$. Then, via Lemma \ref{Kubilius}, choose $p\asymp n^{1/(m+1)}$ such that $\theta_p= \theta_1 + O(\log n^{-100})$. Finally, bearing in mind that the continuity of $F_k$ depends only on $m$ and $l$,  choose $m$ slow enough such that $\min_k|F_k(\theta_p)|\geq \delta_2/2$ and $\delta_2/2> 3 c(n)$. With this choices of $p,q$ and $m$, we have 
	\begin{align}
	&|\hat{\mu_n}(4)- \omega|\leq \epsilon &  \min_k|S_k|> c(n) \nonumber
	\end{align}
	as required. 
\end{proof}

\subsection{Small $3$ spectral quasi-correlations}
\label{small spectral} 
In this section, we construct a sequence of $n\in S'$ such that $\mathcal{Q}(l,n,\exp(-c\sqrt{\log n}))\neq \emptyset$, for some absolute constant $c>0$. We pick integers of the form $n=p_1\cdot p_2 \cdot p_3$, where $p_i\equiv 1 \pmod 4$. Let $\theta_{p_i}$ be the angle of the Gaussian prime above $p_i$, then, by Lemma \ref{Kubilius}, we can choose $p_i\asymp n^{1/3}$ so that
\begin{align}
\theta_1= 0 + O\left( \exp(-c_1\sqrt{\log n})\right) \nonumber \\
\theta_2= \frac{\pi}{3} + O\left( \exp(-c_1\sqrt{\log n})\right) \nonumber \\
\theta_3= -\frac{\pi}{3}+ O\left( \exp(-c_1\sqrt{\log n})\right) \nonumber
\end{align}
for some $c_1>0$. Thus, we have the representations 
\begin{align}
&\xi_1= \exp(i(\theta_{p_1}+\theta_{p_2}+\theta_{p_3}))= 1 +  O( \exp (-c\sqrt{\log n})) \nonumber\\
&\xi_2= \exp(i(\theta_{p_1}+\theta_{p_2}-\theta_{p_3}))=\exp(i 2\pi/3) +  O( \exp (-c\sqrt{\log n})) \nonumber\\
&\xi_3= \exp(i(-\theta_{p_1}+\theta_{p_2}-\theta_{p_3}))=\exp(-i 2\pi/3)  +  O( \exp (-c\sqrt{\log n})).\nonumber
\end{align}
Hence,
\begin{align}
|\xi_1+ \xi_2+\xi_3|n^{-1/2}\ll  \exp (-c\sqrt{\log n}). \nonumber
\end{align}
\subsection{Kac-Rice premises}
We have the following formula for the variance of $\mathcal{L}(f_n,s)$, see \cite[Lemma 3.1, Lemma 3.4, Lemma 3.5 and and page 16]{BMW}. 
\begin{lem}
	\label{Kac-Rice}
	Let $s>0$ and write $r=r_n(\cdot)$ as in \eqref{covariance}. Then, we have 
	\begin{align}
	\Var(\mathcal{L}(f_n,s))= \frac{n}{2}\int_{B(s)\times B(s)}\left(L_2(x-y)+ \epsilon(x-y)\right)dxdy + O\left( \int_{B(s)\times B(s)} r(x-y)^6dxdy\right) \nonumber
	\end{align}
	where 
	\begin{align}
	&8 L_2(x)= r^2 + \Tr(X)+ \frac{\Tr(Y^2)}{4} + \frac{3}{4}r^4 - \frac{\Tr(XY^2)}{8} - \frac{\Tr(X^2)}{16}+ \frac{\Tr(Y^4)}{128} \nonumber \\
	&+\frac{\Tr(Y^2)^2}{256} - \frac{\Tr(X)\Tr(Y^2)}{16}+ \frac{r^2 \Tr(X)}{2}+ \frac{r^2 \Tr(Y^2)}{8}  \nonumber \\
	&| \epsilon(x)|\ll r^6 + \Tr(X^3)+ \Tr(Y^6) \nonumber
	\end{align}
	and 
	\begin{align}
	&X=- \frac{2}{n(1-r^2)}(\nabla r)^t\nabla r &Y=-\frac{2}{n}\left(H + \frac{r}{1-r^2}(\nabla r)^t\nabla r\right)\nonumber
	\end{align}
	where $H$ is the Hessian of $r$, that is $H_{ij}= \frac{\partial^2}{\partial x_i \partial x_j} r$. 
\end{lem}
To evaluate the integrals in Lemma \ref{Kac-Rice}, we need the following lemma. 
\begin{lem}
	\label{moments}
	Let $A= \frac{29}{6} \log 2$, $\epsilon>0$, $n\in S$ and $s>n^{-1/2}(\log n)^{A+\epsilon/2}$. Suppose that  $\mathcal{Q}(l,n, n^{1/2}/(\log n)^{\frac{9}{2}\log 2+\epsilon})=\emptyset$ for $l=2,4,6$, and $N= (\log n)^{\frac{\log 2 }{2}\pm \epsilon/4}$,  then 
	\begin{align}
	& \int_{B(s)\times B(s)} r(x-y)^2 dxdy= \frac{(\pi s^2)^2}{N}\left(1+o\left( \frac{1}{N^2}\right)\right) \nonumber \\
	& \int_{B(s)\times B(s)} r(x-y)^4 dxdy= \frac{3(\pi s^2)^2}{N^2}\left(1+o\left( \frac{1}{N^2}\right)\right)\nonumber \\
	& \int_{B(s)\times B(s)} r(x-y)^6 dxdy= o \left(\frac{s^4}{N^2}\right) . \nonumber
	\end{align}	 
\end{lem}
\begin{proof}
	Let $l$ be either $2$,$4$ or $6$. A direct computation gives
	\begin{align}
	\int_{B(s)\times B(s)} r(x-y)^l dxdy&= \frac{1}{N^l}\int_{B(s)\times B(s)}  \sum_{\xi_1,..,\xi_l} e(\langle \xi_1+...+\xi_l, x-y\rangle) dxdy \nonumber \\
	&= \frac{(\pi s^2)^2}{N^l}\sum_{\xi_1+..+\xi_l=0}1 + O\left( \frac{1}{N^l} \sum_{|\xi_1+..+\xi_l|>0} \left|\int_{B(s)} e(\langle \xi_1+...+\xi_l, x\rangle)dx \right|^2\right) \label{23}
	\end{align}
	The first term on the right hand side of \eqref{23} is equal to $(\pi s^2)^2/N$ if $l=2$, $3(\pi s^2)^2/N^2$ if $l=4$ and if $l=6$ we use \eqref{2} to see that it is bounded by $O(s^4/N^{5/2})$. Thus, we are left with bounding the second term on the right of \eqref{23}. Carrying out the integral gives
	\begin{align}
	\frac{1}{N^l}\sum_{|\xi_1+..+\xi_l|>0} \left|\int_{B(s)} e(\langle \xi_1+...+\xi_l, x\rangle)dx \right|^2 = 	\frac{(\pi s^2)^2}{N^l}\sum_{|\xi_1+..+\xi_l|>0} \left| \frac{J_1(s|\xi_1+...+\xi_l|)}{s|\xi_1+...+\xi_l|}\right|^2  \label{24} 
	\end{align}
	By assumption $\mathcal{Q}(l,n, n^{1/2}/(\log n)^{9\log 2/2+\epsilon})=\emptyset$, thus  $s|\xi_1+...+\xi_l|\geq (\log n)^{\frac{1}{3}\log2 -\epsilon/2}$, bearing in mind that $J_1(T)/T\ll T^{-3/2}$ and $N= (\log n)^{\frac{\log 2 }{2}\pm \epsilon/4}$, we obtain 
	\begin{align}
	\text{RHS}\eqref{24}\ll \frac{1}{(\log n)^{\log2+3\epsilon/2}}=o(N^{-2}) \nonumber
	\end{align}
	as required. 
\end{proof}
Using Lemma \ref{Kac-Rice} and following similar calculations to Lemma \cite[Lemma 3.4]{BMW}, we obtain the following lemma. 
\begin{lem}
	\label{calculations}
	Under the assumptions of Lemma \ref{moments}, we have 
	\begin{align*}
	&	\int_{B(s)\times B(s)} \Tr X(x-y)dxdy= (\pi s^2)^2 \left( \frac{-2}{N}- \frac{2}{N^2} + o\left( \frac{1}{N^2}\right)\right) \nonumber \\
	&\int_{B(s)\times B(s)} \Tr Y(x-y)^2dxdy= (\pi s^2)^2 \left( \frac{4}{N}- \frac{4}{N^2} + o\left( \frac{1}{N^2}\right)\right) \nonumber \\
	&\int_{B(s)\times B(s)} \Tr X(x-y)Y(x-y)^2dxdy= (\pi s^2)^2 \left( - \frac{4}{N^2} +o\left( \frac{1}{N^2}\right)\right) \nonumber \\
	&\int_{B(s)\times B(s)} \Tr X(x-y)^2dxdy= (\pi s^2)^2 \left(  \frac{8}{N^2} +o\left( \frac{1}{N^2}\right)\right) \nonumber \\
	\end{align*}
	\begin{align*}
	&\int_{B(s)\times B(s)} \Tr Y(x-y)^4dxdy= (\pi s^2)^2 \left(  \frac{2(11+ \hat{\mu}(4)^2)}{N^2} + o\left( \frac{1}{N^2}\right)\right) \nonumber \\
	&\int_{B(s)\times B(s)} (\Tr Y(x-y)^2)^2dxdy= (\pi s^2)^2 \left(  \frac{4(7+ \hat{\mu}(4)^2)}{N^2} + o\left( \frac{1}{N^2}\right)\right) \nonumber \\
	&\int_{B(s)\times B(s)} \Tr X(x-y) \Tr Y(x-y)^2dxdy= (\pi s^2)^2 \left( - \frac{8}{N^2} + o\left( \frac{1}{N^2}\right)\right) \nonumber \\
	&\int_{B(s)\times B(s)} r(x-y)^2 \Tr X(x-y)dxdy= (\pi s^2)^2 \left( - \frac{2}{N^2} + o\left( \frac{1}{N^2}\right)\right) \nonumber \\
	&\int_{B(s)\times B(s)} r(x-y)^2 \Tr Y(x-y)^2dxdy= (\pi s^2)^2 \left(  \frac{8}{N^2} +o\left( \frac{1}{N^2}\right)\right) \nonumber \\
	&\int_{B(s)\times B(s)} \Tr Y(x-y)^3dxd= o\left( \frac{s^4}{N^2}\right) \nonumber \\
	&\int_{B(s)\times B(s)} \Tr Y(x-y)^6dxd= \ o\left( \frac{s^4}{N^2}\right) \nonumber .
	\end{align*}
\end{lem}
\subsection{Concluding the proof of Theorem \ref{thm 2}}
We are finally ready to prove Theorem \ref{thm 2}.
\begin{proof}[Proof of Theorem \ref{thm 2}]
	By Proposition \ref{limit point} and Theorem \ref{theorem 1}, there exists a density one   subsequence of integers $n\in S$ such that both $(1)$ and Lemma \ref{calculations} hold. For such sequence we can evaluate the variance using Lemma \ref{Kac-Rice} (and following identical calculation to the proof of \cite[Theorem 1.1]{BMW}), so $(2)$ follows. To prove $(3)$ we again resort to calculations in \cite{BMW}: for all $s>0$ we have 
	\begin{align}
	\Cov(\mathcal{L}(f_n,s), \mathcal{L}(f_n))= (\pi \cdot s^2)^2 \Var(\mathcal{L}(f_n)). \label{25}
	\end{align}
	Hence, combining \eqref{25}, part $(2)$ and \eqref{variance f_n} we obtain  $(3)$. 
\end{proof}
\section{Proof of Theorem \ref{thm 3}}
To prove Theorem \ref{thm 3}, we use the stability of the nodal set under small perturbations, as in \cite{NS}. To prove Theorem \ref{thm 3} will need a series of results. 
\subsection{Stability of the nodal set}
The following deterministic lemma, inspired by \cite[Lemma 4.7]{BW}, will be our main tool in studying small perturbations of the nodal set of $f_n$. 
\begin{lem}
	\label{stability}
	Let $h, \vartheta: B(1)\rightarrow \mathbb{R}$ be two  smooth functions and assume the following:
	\begin{enumerate}
		\item For some $\beta>0$ we have 
		\begin{align}
		\underset{y\in B(1)}{\min}\max\{|h|,|\nabla h|\}> \beta. \nonumber
		\end{align}
		\item For some $M>0$ we have 
		\begin{align}
		||h||_{C^2},||h+ \vartheta||_{C^2}<M. \nonumber
		\end{align}
		\item For some $\tau>0$ we have 
		\begin{align}
		||\vartheta||_{C^2}<\tau. \nonumber
		\end{align}
	\end{enumerate}
	Then, provided that $\tau\leq \beta^2 \cdot(16M)^{-1}$, we have 
	\begin{align}
	\mathcal{L}(h+\vartheta) =  \mathcal{L}(h) \left(1+ O\left(\tau \frac{M^3}{\beta^4 }\right)\right). \nonumber
	\end{align}
\end{lem}
\begin{proof}
	Let $\gamma_h$ be a connected component of $h^{-1}(0)$, fix some $z_0\in \gamma_h$, and let $N(z_0)=\nabla h/ |\nabla h| $ be a unit normal vector of $\gamma_h$ at $z_0$. By assumption (1), $|\nabla h(z_0)|>\beta$, thus, bearing in mind that all second derivatives of $h$ are bounded, we can find some $0<r_0=r_0(\beta,M)\leq \beta/4 M$ such that
	\begin{align}
	N(z_0) \cdot \nabla h(z)>\beta/2 \label{31}
	\end{align}
	for all $z$ in a $r_0$-neighbourhood of $z_0$. Now, consider the function
	\begin{align}
	\zeta(r)=h(z_0+rN(z_0)) + \vartheta(z_0+rN(z_0)).\nonumber
	\end{align}
	Bearing in mind assumption $(3)$ and using \eqref{31}, we have 
	\begin{align}
	\zeta'(r)>\beta/2-\tau>\beta/4  \label{derivative growing}
	\end{align}
	for $\tau< \beta/4$ and uniformly for all $|r|< r_0$. Now, suppose that $0<\zeta(0)= \vartheta(z_0)\leq \tau$ and that $4\tau/\beta\leq r_0$ then, bearing in mind \eqref{derivative growing}, the mean value theorem implies that there exist a unique (negative) $r=r(z_0)$ with $|r|\leq 4\tau/\beta \leq r_0$, that is $\tau\leq \beta^2 \cdot (16 M)^{-1}$,  such that $\zeta(r(z_0))=0$. Arguing similarly in the case $\zeta(0)<0$, and taking $r=0$ if $\zeta(0)=0$, we find that the map 
	\begin{align}
	z\rightarrow z+r(z)N(z) \label{32}
	\end{align}
	is an injection of $\gamma_h$ into $\gamma_{h+\vartheta}$ (where $\gamma_{h+\vartheta}$ is a connected component of $(h+\vartheta)^{-1}(0)$), provided $\tau\leq  \beta^2 \cdot (16M)^{-1}$.
	
	\begin{claim}
		\label{main claim}
		Via the implicit function theorem, parametrize $\gamma_h$ in some neighbourhood $U$ around $z_0$ as $C(t)=(t,q(t))$ for some smooth function $q:U\rightarrow \mathbb{R}$, then 
		\begin{align}
		\frac{d}{dt}r(C(t))= \nabla r C' \ll \tau M^3/\beta^4 \nonumber
		\end{align}
	\end{claim} 
	\begin{proof}  By (\ref{32}),  we have the following system:
		\begin{align}	
		&(h+\vartheta)(C(t)+ r(C(t))N(C(t))) =0 &h(C(t))=0	\nonumber
		\end{align}
		Taking the derivative with respect to $t$, we obtain 
		\begin{align}
		&\nabla (h+\vartheta)(C+ r N) \cdot [C'+ ( \nabla r  C')N+ r(N' C')]=0 \label{40} \\ 
		&\nabla h(C)C'=0 \nonumber  
		\end{align}
		By the implicit function theorem $|C'|\leq M/\beta$, thus, bearing in mind assumption (2) and the definition of $N$, $|N'C'|\leq M^2/\beta^2$. Moreover, by construction $r\leq 4\tau/\beta$, $|N|\leq 1$, and by assumption (2) $|\nabla (h+ \vartheta)|\leq M$, therefore we can re-write \eqref{40} as
		\begin{align}
		&\nabla (h+\vartheta)(C+ r N) \cdot  (\nabla r C')N= - \nabla(h+\vartheta)(C+ r N) \cdot C' + O\left( \tau \frac{M^3}{\beta^3}\right) \label{30}  \\
		&\nabla h(C)C'=0. \nonumber
		\end{align}
		Note that, by assumptions $(2)$ and $(3)$, for $r<1$ say, we have
		\begin{align}
		\nabla( h + \vartheta)(C+ rN) = \nabla h(C) + O\left(\frac{M}{\beta} \cdot \tau\right). \label{C2 bound}  
		\end{align}
		Thus, using the expansion \eqref{C2 bound} on the right hand side of \eqref{30} and subtracting the second equation from the first, we have 
		\begin{align}
		\nabla (h+\vartheta)(C+ r N) \cdot  (\nabla r C') \ll \tau  \frac{M^3}{\beta^3}. \label{41}
		\end{align}
		As $\nabla (h+\vartheta)\geq \beta/2$, the claim follows from \eqref{41}. 
	\end{proof} Finally, fix some $z_0\in \gamma_h$ and let $U$ be as in Claim \ref{main claim}, moreover let $V$ be the image of $U$ under the map  \eqref{32} so that $V$ is parametrised by $C(t)+ r(C(t))\cdot N(C(t))=: \tilde C(t)$. Now, Claim \ref{main claim} implies that $|C'(t)|= |\tilde{C}'(t)| + O \left(\tau M^3 \beta^{-4}\right)$, thus  
	\begin{align}
	\mathcal{L}(h\lvert_U)&=	\int_U |C'(t)|dt = \int_V |\tilde{C}'(t)|dt +  O\left(\tau \frac{M^3}{\beta^4 }\right)\nonumber \\
	&=\mathcal{L}((h+ \vartheta)\lvert_V) +O\left(\tau \frac{M^3}{\beta^4 }\right). \label{1/2}
	\end{align}
	Summing \eqref{1/2} over the zero set of $h$ we obtain the required result. 
\end{proof}
\subsection{Quantifying $M$}
By \cite[Lemma 3.12]{MV} and \cite[Corollary 2.2]{BM}, we have the following result:
\begin{lem}
	\label{lemma2}
	Let $f_n$ be as in \eqref{function}, $R>1$ $k\geq 0$, and let $F_n(y)= f_n(Ry/\sqrt{n})$ for $y\in B(1)$. Then we have 
	\begin{align}
		\mathbb{P}\left( ||F_n||_{C^k(B(1))}\gg R^{k+1}\log R\right)\leq e^{-C(\log R)^2}  \nonumber
	\end{align}
	for some $C>0$.
\end{lem}
\subsection{Quantifying $\beta$}
\begin{lem}
	\label{lemma 1}
	Let $f_n$ be as in \eqref{function}, $R>1$  and $F_n(y)= f_n(Ry/\sqrt{n})$ for $y\in B(1)$. Suppose that $n\in S$ satisfies the conclusion of Theorem \ref{Erdos-Hall}, then  
	\begin{align}
	\mathbb{P}\left(\min_{y\in B(1)}\max\{|F_n(y)|,|\nabla F_n(y)|\}\leq R^{-4}(\log R)^{-3}\right)\leq (\log R)^{-1} \nonumber	\end{align}
\end{lem}
\begin{proof}
	Let $\epsilon_1>0$ be some (small) parameter. Differentiating $r_n(y,y)=1$, we see that $F_n(y)$ and $\nabla F_n(y)$ are independent random variables. Therefore, bearing in mind that $F_n$ is a stationary field, we have 
	\begin{align}
	\mathbb{P}\left( |F_n(y)|\leq \epsilon_1 \hspace{2mm} \text{and} \hspace{2mm} |\nabla F_n(y)|\leq \epsilon_1 \right)= \mathbb{P}\left( |F_n(y)|\leq \epsilon_1 \right) \mathbb{P}\left( |\nabla F_n(y)|\leq \epsilon_1 \right) \nonumber \\
	=  \mathbb{P}\left( |F_n(0)|\leq \epsilon_1\right)\mathbb{P}\left(|\nabla F_n(0)|\leq \epsilon_1 \right)= \frac{1}{\sqrt{2\pi}}\int_{-\epsilon_1}^{\epsilon_1}e^{-t^2/2} dt \cdot \mathbb{P}\left(|\nabla F_n(0)|\leq \epsilon_1 \right) \nonumber \\
	\leq \epsilon_1 \mathbb{P}\left(|\nabla F_n(0)|\leq \epsilon_1 \right) \label{27}
	\end{align}
	The covariance matrix of the Gaussian vector $\nabla F_n(0)$ is given by
	\begin{align}
	C(0)=\frac{	4\pi^2 R^2}{N}\begin{pmatrix}
	\sum_{\xi} \xi_1^2 & \sum_{\xi} \xi_1\xi_2 \\
	\sum_{\xi} \xi_1\xi_2 & \sum_{\xi} \xi_2^2
	\end{pmatrix} \nonumber
	\end{align}  
	with determinant
	\begin{align}
	\det C(0)= \frac{16 \pi ^4 R^4}{N^2}\sum_{\xi,\xi'}(\xi_1^2(\xi'_2)^2- \xi_1\xi_2\xi'_1\xi'_2)= \frac{16 \pi ^4 R^4}{N^2} \sum_{\xi,\xi'} \sin( \theta_{\xi}-\theta_{\xi'})^2 \nonumber
	\end{align}
	where $\xi= e^{2\pi i\theta_{\xi}}$. Using Theorem \ref{Erdos-Hall} to pass from the sum to the integral, and the identity $\sin (\cdot)^2=(1-\cos(2\cdot))/2$, we have 
	\begin{align}
	\det C(0)=8 \pi ^4 R^4 + \int_{0}^{2\pi}\int_{0}^{2\pi}\cos (2(x-y))dxdy+o(1)= 8 \pi ^4 R^4 +o(1) \nonumber
	\end{align}
	Hence, 
	\begin{align}
	\mathbb{P}\left(|\nabla F_n(0)|\leq \epsilon_1 \right)&=\frac{1}{4\sqrt{2}\pi^4 R^2 } \int_{[-\epsilon_1,\epsilon_1]^2}\exp \left((\det C(0))^{-1} \sum_{\xi} \xi_2^2 x_1^2 - 2\xi_1\xi_2 x_1x_2+ \xi_1^2 x_2^2\right)dx_1dx_2 \nonumber \\ &\leq    \frac{\epsilon_1^2}{R^2 } \label{28}
	\end{align}
	Inserting \eqref{28} into \eqref{27}, we obtain 
	\begin{align}
	\mathbb{P}\left( |F_n(y)|\leq \epsilon_1 \hspace{2mm} \text{and} \hspace{2mm} |\nabla F_n(y)|\leq \epsilon_1 \right) \leq \frac{\epsilon_1^3}{R^2}. \label{29}
	\end{align}
	Now, consider an $\eta$-net on $B(1)$ and denote by $y_i$ the points of the net. By Lemma \ref{lemma2}, we know that $ ||F_n||_{C^2(B(1))}\ll R^3\log R$ outside an event of probability at most $O(e^{-C(\log R)^2})$. Thus, since every point $y\in B(1)$ is at distance at most $\eta$ from a point $y_i$ on the net, we have 
	\begin{align}
	&F_n(y)= F_n(y_i) + O(\eta  R^3 \log R  )  & \nabla F_n(y)= \nabla F_n(y_i) + O(\eta  R^3\log R  ), \nonumber
	\end{align}
	for some $y_i$. Therefore, if $|F_n(y)|\leq \beta$ and $|\nabla(F_n(y))|\leq \beta$, then, taking $\eta= c\beta( R^3 \log R)^{-1}$ for some sufficiently small $c>0$, also $|F_n(y_i)|\leq \beta/2$ and $|\nabla(F_n(y_i))|\leq \beta/2$, which has probability at most $O(\beta^3/R^2)$ by \eqref{29}. Taking the union bound over the net, which has $O(\eta^{-2})$ points, we deduce that
	 $$\mathbb{P}\left(\min_{y\in B(1)}\max\{|F_n(y)|,|\nabla F_n(y)|\}\leq \beta\right)\ll\frac{\beta^3}{R^2}\eta^{-2} +e^{-C (\log R)^2}\ll \beta R^4 \log R^2 +e^{-C (\log R)^2}.$$ Hence, taking $\beta= R^{-4}(\log R)^{-3}$ we deduce the lemma. 
\end{proof}

\subsection{Quantifying $\tau$}

To quantify $\tau$ we use the following recent result \cite[Theorem 5.5]{BM}:
\begin{lem}[Beliaev-Maffucci]
	\label{lemma 3}
	Let $R>1$, $\epsilon>0$, $n\in S$, $F_n(y)=f(Ry/\sqrt{n})$ and $F_{\mu}(y)=f_{\mu}(Ry)$  for $y\in B(1)$. Suppose that $n$ satisfies the conclusion of Theorem \ref{Erdos-Hall}, then there exists a coupling such that the field $\tilde{F}= F_n- F_{\mu}$ satisfies
	\begin{align}
	\mathbb{P}\left(||\tilde{F}||_{C^2(B(1))} \gg R^2 \log R \cdot (\log n)^{-2\kappa/3+\epsilon} \right) \ll (\log R)^{-1}. \nonumber
	\end{align}
\end{lem}

\subsection{Proof of Theorem \ref{thm 3}}
We are finally ready to prove Theorem \ref{thm 3}
\begin{proof}[Proof of Theorem  \ref{thm 3}]
	Take a subsequence of $n\in S$ such that the conclusion of Theorem \ref{Erdos-Hall} holds (and $N\rightarrow \infty$). Let $R>1$, rescale $f_n$ as $F_n(y)=f_n(Ry/\sqrt{n})$ and $f_{\mu}$ as  $F_{\mu}(y)=f_{\mu}(Ry)$ for  $y\in B(1)$. By Lemma \ref{lemma 1}, Lemma \ref{lemma2} and Lemma \ref{lemma 3}, outside an event of probability at most $O((\log R)^{-1})$, we have the following bounds: 
	\begin{enumerate}
		\item \begin{align}
		\min_{y\in B(1)}\max\{|F_n|,|\nabla F_n|\} \geq R^{-4}(\log R)^{-3}:= \beta\nonumber
		\end{align}
		\item \begin{align}
		||F_n||_{C^2} \leq 2 R^3\log R :=M\nonumber
		\end{align}
		\item \begin{align}
		||F_{\mu}- F_n||_{C^2}\leq R^2 \log R (\log n)^{-2\kappa/3 +\epsilon}:= \tau \nonumber
		\end{align}
	\end{enumerate}
	Lemma \ref{stability}, provided that $\tau \ll \beta ^2/ M= R^{-11} (\log R)^{-5}$,  implies that
	\begin{align}
	\mathcal{L}(F_\mu)= \mathcal{L}(F_n) \left( 1 +O\left(\tau \frac{M^3}{\beta^4}\right)\right) = \mathcal{L}(F_n)\left( 1 +O\left(R^{27}\log R^{16} \cdot (\log n)^{-2\kappa/3 +\epsilon} \right) \right) \nonumber
	\end{align} 
	outside an event of probability at most $O(\log R^{-1})$. The Kac-Rice formula  \cite[Theorem 6.3]{AW} implies that $\mathbb{E}[\mathcal{L}(F_n)]\ll R$. Thus, outside and event of probability at most $O(\log R^{-1})$, $\mathcal{L}(F_n)\ll R\log R$ which implies
	\begin{align}
	|\exp( it\mathcal{L}(F_n)- \exp(it \mathcal{L}(F_{\mu}))| \ll t | \mathcal{L}(F_n)- \mathcal{L}(F_{\mu})| \leq t R^{28}  (\log R)^{17} (\log n)^{-2\kappa/3 +\epsilon}. \label{42}
	\end{align}
	Taking $R= (\log n)^{\frac{\kappa}{42}-\epsilon/2}$ and $n$ large enough depending on $t$, \eqref{42} implies the Theorem. 
\end{proof}
\section*{Acknowledgement}
The author would like to thank Igor Wigman for pointing out the question considered here and for the many discussions, Zeev Rudnick for valuable comments that helped improving the presentation of the article, as well as Oleksiy Klurman for useful conversations. The author would also like to thank the anonymous referees for pointing our an error in the previous draft of the article and their valuable comments that greatly helped to improve the presentation.  This work was supported by the Engineering and Physical Sciences Research Council [EP/L015234/1].  
The EPSRC Centre for Doctoral Training in Geometry and Number Theory (The London School of Geometry and Number Theory), University College London.

\end{document}